\documentclass[12pt]{article}
\usepackage{amsmath, amssymb, amsfonts}
\usepackage{amsthm}
\usepackage[utf8]{inputenc}
\usepackage{csquotes}
\usepackage{geometry}
\usepackage{adjustbox}
\usepackage{graphicx}
\usepackage{verbatim}
\usepackage{array}
\usepackage{enumitem}
\usepackage{booktabs}
\usepackage{hyperref}
\usepackage{url}
\usepackage[capitalize]{cleveref}
\usepackage{tikz}
\usetikzlibrary{positioning, shapes.geometric, arrows.meta}

\title{Stratified Algebra}
\author{Stanislav Semenov \\
\href{mailto:stas.semenov@gmail.com}{stas.semenov@gmail.com} \\
\href{https://orcid.org/0000-0002-5891-8119}{ORCID: 0000-0002-5891-8119}}
\date{May 24, 2025}

\theoremstyle{definition}
\newtheorem{definition}{Definition}[section]

\newtheorem{example}{Example}[section]

\theoremstyle{plain}

\newtheorem{corollary}[definition]{Corollary}
\newtheorem{proposition}[definition]{Proposition}

\theoremstyle{remark}
\newtheorem*{remark}{Remark}

\begin{document}

\maketitle

\begin{abstract}
We introduce and investigate the concept of \emph{Stratified Algebra}, a new algebraic framework equipped with a layer-based structure on a vector space. We formalize a set of axioms governing intra-layer and inter-layer interactions, study their implications for algebraic dynamics, and present concrete matrix-based models that satisfy different subsets of these axioms. Both associative and bracket-sensitive constructions are considered, with an emphasis on stratum-breaking propagation and permutation symmetry. This framework proposes a paradigm shift in the way algebraic structures are conceived: instead of enforcing uniform global rules, it introduces stratified layers with context-dependent interactions. Such a rethinking of algebraic organization allows for the modeling of systems where local consistency coexists with global asymmetry, non-associativity, and semantic transitions.
\end{abstract}

\subsection*{Mathematics Subject Classification}
17A30 (Algebras satisfying identities), 17A36 (Automorphisms and endomorphisms of nonassociative algebras), 15A75 (Structure of general matrix rings)

\subsection*{ACM Classification}
F.4.1 Mathematical Logic, G.1.0 Numerical Analysis

\section*{Introduction}

Classical algebraic structures—such as groups, rings, and associative algebras—are built on the assumption of global uniformity: operations obey the same rules across the entire space, ensuring properties like associativity, commutativity, or distributivity hold universally \cite{Lang2002,Artin2010,Bourbaki1998,Rotman2015,Jacobson2009}. Yet in many complex systems, algebraic behavior is not globally consistent. Instead, operations may be locally well-behaved—associative or commutative within certain subsets—while globally exhibiting non-associativity, asymmetry, or semantic discontinuity. Such behavior calls for a new paradigm: one that preserves \emph{local algebraic coherence} while allowing for \emph{global structural rupture}.

This paper introduces such a paradigm through the formal notion of \emph{Stratified Algebra}—a layered algebraic structure on a vector space, where binary operations behave differently within and across disjoint strata. In our stratified algebra, we show how simple algebraic rules at the local level can give rise to complex global dynamics, including transitions between strata, noncommutative flows, and bracket-sensitive expressions. In particular, we demonstrate how distinct computational paths—starting and ending at the same point—can traverse different strata depending on how the expression is bracketed associatively. This captures an essential feature of stratified behavior: multiple equivalent endpoints connected by semantically distinct trajectories.

To formalize this, we introduce a system of axioms governing intra-stratum associativity, cross-layer asymmetry, layered permutation symmetry, and bracketing sensitivity. These axioms encode the algebraic mechanisms by which local order can coexist with global nonlinearity, and how structured interactions between strata lead to emergent semantics.

We construct explicit realizations of such algebras via matrix-generated multiplications and associated structure tensors. These models illustrate the layered behavior in concrete terms and allow for systematic exploration of associativity, commutator structure, and stratified propagation, extending beyond classical frameworks into non-associative settings \cite{Schafer1995}. The framework supports both associative and non-associative regimes and reveals how algebraic transitions between layers can be encoded through algebraic operations themselves.

Stratified Algebra provides a novel formalism for modeling systems where layered context, dynamic interaction, and nonuniform rules play a central role. Potential applications range from symbolic computation and multi-context semantics to cluster-based dynamics and algebraic representations of hierarchical systems.

\section{Definition: Stratified Algebra}

Let \( K \) be a field and let \( V = K^n \setminus \{0\} \) be a finite-dimensional vector space over \( K \), with the zero vector excluded. A \emph{Stratified Algebra} on \( V \) consists of the following data:

\begin{itemize}
    \item A disjoint decomposition
    \[
    V = \bigsqcup_{\alpha \in \Lambda} V_\alpha,
    \]
    where \( \{V_\alpha\}_{\alpha \in \Lambda} \) is a family of pairwise disjoint nonempty subsets of \( V \), indexed by a set \( \Lambda \). The strata \( V_\alpha \) are not assumed to be subspaces, and \( \Lambda \) may be infinite or even uncountable.

    \item A binary operation \( * : V \times V \to V \), defined by an algebraic rule
    \[
    a * b := f(a, b),
    \]
    where \( f : V \times V \to V \) is a map that may be bilinear, polynomial, or more generally algebraic. In particular, we do not require that \( f(a, b) \) be representable in the form \( M_a \cdot b \) for some matrix \( M_a \); matrix-based constructions are treated as special cases.
\end{itemize}

The operation \( * \) must satisfy the following axioms:

\subsection*{(SA1) Local Algebraic Structure}

For every stratum \( V_\alpha \) in the decomposition \( V = \bigsqcup_{\alpha \in \Lambda} V_\alpha \), the operation \( * \) restricted to \( V_\alpha \times V_\alpha \) is both commutative and associative. That is, for all \( a_1, a_2, a_3 \in V_\alpha \), we have:
\[
a_1 * a_2 = a_2 * a_1,
\]
\[
(a_1 * a_2) * a_3 = a_1 * (a_2 * a_3).
\]
In particular, each stratum \( V_\alpha \) forms a commutative associative magma under the restricted operation.

\subsection*{(SA2) Cross-Layer Asymmetry}

Let \( a \in V_\alpha \), \( b \in V_\beta \) with \( \alpha \ne \beta \). Then:

\begin{itemize}
    \item The product lies outside both operand strata:
    \[
    a * b \notin V_\alpha \cup V_\beta.
    \]

    \item There exists a stratum \( V_\gamma \notin \{V_\alpha, V_\beta\} \) such that
    \[
    a * b \in V_\gamma.
    \]

    \item The operation is non-commutative across strata:
    \[
    a * b \ne b * a.
    \]

    \item The operation is non-associative across strata in general: there exists \( c \in V_\delta \) such that
    \[
    (a * b) * c \ne a * (b * c).
    \]
\end{itemize}

\subsection*{(SA3) Layered Permutation Symmetry}

Let \( a_1, \ldots, a_m \in V_\alpha \), \( b \in V_\beta \) with \( \alpha \ne \beta \). For any permutation \( \sigma \in S_m \), the fully left-associated product satisfies the \emph{Layered Permutation Symmetry} property:
\[
b * a_1 * \cdots * a_m = b * a_{\sigma(1)} * \cdots * a_{\sigma(m)},
\]
where the left-associated product is defined recursively as
\[
b * a_1 * \cdots * a_m := (\cdots((b * a_1) * a_2) \cdots ) * a_m.
\]

\begin{definition}[Layered Permutation Symmetry Operator]
Let \( a \in V_\alpha \), \( b, c \in V_\beta \) with \( \alpha \ne \beta \). The \emph{Layered Permutation Symmetry operator} is the element
\[
\mathrm{LPS}(a,b,c) := (a * b) * c - (a * c) * b.
\]
This expression measures the failure of permutation symmetry in the left-associated action by elements from the stratum \( V_\beta \) on an element of \( V_\alpha \).
\end{definition}

\begin{proposition}
The following statements are equivalent:
\begin{enumerate}
\item Axiom SA3 holds for the algebra.
\item \(\mathrm{LPS}(a,b,c) = 0\) for all \( a \in V_\alpha \), \( b, c \in V_\beta \) with \( \alpha \ne \beta \).
\item The left-associated action of \( V_\beta \) on \( V_\alpha \) is symmetric under permutations of the acting elements.
\end{enumerate}
\end{proposition}

\begin{remark}
In concrete realizations, this symmetry can be checked computationally for small values of \( m \). The LPS operator provides an analytical tool for studying such symmetry in general algebraic frameworks.
\end{remark}

\subsection*{(SA4) Stratum-Bracketing Sensitivity}

Let $a_1, \ldots, a_m \in V_\alpha$, $b \in V_\beta$ with $\alpha \neq \beta$. For any nontrivial bracketing of the form
\[
(b \ast a_1 \ast \cdots \ast a_l) \ast (a_{l+1} \ast \cdots \ast a_m), \quad 1 \leq l < m - 1,
\]
the following conditions must hold:

\begin{enumerate}
\item \textbf{Non-associativity}:
The bracketed product differs from the fully left-associated product:
\[
(b \ast a_1 \ast \cdots \ast a_l) \ast (a_{l+1} \ast \cdots \ast a_m) \neq b \ast a_1 \ast \cdots \ast a_m
\]

\item \textbf{Stratum projection}:
The bracketed product lands in a new stratum:
\[
(b \ast a_1 \ast \cdots \ast a_l) \ast (a_{l+1} \ast \cdots \ast a_m) \in V_\gamma
\]
where $\gamma \neq \alpha,\beta$ and $V_\gamma$ is distinct from the original strata.

\end{enumerate}

This axiom captures how bracketing structure induces transitions between algebraic strata, serving as a mechanism for semantic transformation.

\subsection*{Hierarchies of Stratified Algebras}

The axioms \emph{(SA1)}–\emph{(SA4)} define progressively stronger structural constraints on the operation \( * \) within a stratified algebra. In particular, the choice of which axioms to enforce determines the expressive and dynamic capabilities of the resulting algebra:

\begin{itemize}
    \item Algebras satisfying only \emph{(SA1)} and \emph{(SA2)} exhibit strict layer-based algebraic behavior: full associativity and commutativity within each stratum, and strict asymmetry across strata. These are \emph{layered associative} systems with discontinuous inter-layer behavior.

    \item If \emph{(SA3)} is also imposed, the algebra supports symmetric multiaction of one stratum on another, suggesting a notion of commutative symmetry in layered influence, even in the presence of global noncommutativity.

    \item The full set of axioms \emph{(SA1)}–\emph{(SA4)} leads to a richer and more expressive class of algebras, where nontrivial reassociation of operands induces semantic shifts across strata. These algebras are \emph{globally stratified} and \emph{bracket-sensitive}, encoding nonlinear propagation, hierarchy formation, and context-dependent computation.
\end{itemize}

Thus, depending on the intended application, one may consider:
\begin{itemize}
    \item \emph{Weak stratified algebras} (satisfying SA1–SA2),
    \item \emph{Symmetric stratified algebras} (satisfying SA1–SA3),
    \item \emph{Fully stratified algebras} (satisfying SA1–SA4).
\end{itemize}

Each class admits its own family of models and internal dynamics. The weaker variants permit globally associative constructions with layer-level predictability, while the full axiom system enables bracket-induced transitions and algebraic nonlocality across strata.

\subsection*{Stratum-Stability vs. Stratum-Breaking}

\begin{definition}[Stratum-Stable and Stratum-Breaking Products]
Let \( a_1, \ldots, a_k \in V \). A product expression of the form
\[
a_1 * a_2 * \cdots * a_k := ( \cdots ((a_1 * a_2) * a_3) \cdots ) * a_k
\]
is said to be \emph{stratum-stable} if its result lies entirely within the union of the strata containing the operands:
\[
a_1 * a_2 * \cdots * a_k \in \bigcup_{i=1}^k V_{\alpha_i}, \quad \text{where } a_i \in V_{\alpha_i}.
\]
Otherwise, the product is called \emph{stratum-breaking}.
\end{definition}

\paragraph{Propagation Dynamics.}
A key structural feature of Stratified Algebras is that any binary product involving operands from distinct strata necessarily produces a result outside the original layers (by axiom SA2). However, if a product such as \( a_1 * a_2 \in V_\gamma \), and further elements \( a_3, \dots, a_k \in V_\gamma \), then the extended product \( a_1 * a_2 * \cdots * a_k \) remains in \( V_\gamma \) — exhibiting local stratum-stability within the new layer.

In contrast, once an additional operand \( a_{k+1} \in V_\delta \), \( \delta \ne \gamma \), is introduced, the resulting product becomes stratum-breaking again. This illustrates a layered propagation mechanism: strata may emerge, stabilize, and then rupture again under the introduction of foreign elements.

\paragraph{Interpretation.}
This distinction reflects whether a sequence of multiplications preserves the semantic layer structure of its inputs or propagates into new strata. Stratum-stable expressions retain internal algebraic consistency, while stratum-breaking expressions indicate semantic transitions between layers.

\paragraph{Structural Role.}
Stratum-breaking behavior underlies the dynamic and generative power of a stratified algebra. It enables new strata to emerge as a consequence of inter-layer interactions or nontrivial reassociations of products. In contrast, stratum-stable behavior reflects closure and internal regularity within a given layer or a compatible cluster of layers.

\paragraph{Stability and Subalgebra Formation.}
A subset \( W \subset V \) is called a \emph{stratum-stable subalgebra} if it satisfies
\[
\forall a_1, \ldots, a_k \in W, \quad a_1 * \cdots * a_k \in W.
\]
That is, \( W \) is closed under repeated application of the operation \( * \), regardless of the bracketing. Importantly, \( W \) may consist of elements drawn from one or several strata \( V_\alpha \), as long as all intermediate and final products remain within \( W \). In this sense, stratum-stable subalgebras generalize the notion of closure to stratified settings and may span across multiple, but compatible, strata. Such subalgebras inherit local associativity where it holds and capture internally consistent regions of algebraic computation.

\paragraph{Depth of Stratified Propagation.}
Given a product expression, one may define the \emph{stratified depth} as the number of distinct strata traversed or generated in the course of evaluating the expression. Stratum-stable expressions have minimal depth, while stratum-breaking ones exhibit stratified growth and complexity.

\paragraph{Remark.}
Stratum-breaking phenomena are central to the expressive strength of stratified algebras. They serve as a mechanism for controlled algebraic nonlinearity, non-associativity, and context-sensitive computation across layers of structure.

\section{The Linear Case}

\subsection{Matrix-Generated Multiplication and Its Tensor Representation}

In this section, we study a particular class of Stratified Algebras in which the matrix \( M_a \in \mathrm{Mat}_{n \times n}(K) \) is defined by a linear combination of fixed matrices, with coefficients depending linearly on the coordinates of the vector \( a \in V \cong K^n \).

\begin{definition}[Matrix Formulation]
Let \( \{E^{(1)}, \dots, E^{(r)}\} \subset \mathrm{Mat}_{n \times n}(K) \) be a fixed collection of basis matrices, and let \( \{\lambda^{(j)} : V \to K \} \) for \( 1 \leq j \leq r \) be linear functionals on \( V \), i.e.,
\[
\lambda^{(j)}(a) = \sum_{i=1}^n \lambda^{(j)}_i a_i,
\]
where \( a = (a_1, \dots, a_n) \in V \), and \( \lambda^{(j)}_i \in K \). Then the matrix \( M_a \) is given by
\[
M_a = \sum_{j=1}^r \lambda^{(j)}(a) \cdot E^{(j)},
\]
and the binary operation is defined by
\[
a * b := M_a \cdot b.
\]
\end{definition}

\begin{definition}[Tensor Formulation]
Equivalently, the multiplication \( * : V \times V \to V \) can be defined by a structure tensor \( \alpha_{ijk} \in K \) such that
\[
(a * b)_k = \sum_{i=1}^n \sum_{j=1}^n \alpha_{ijk} \, a_i \, b_j,
\]
where \( a = (a_1, \dots, a_n) \), \( b = (b_1, \dots, b_n) \in V \).
\end{definition}

\begin{proposition}
The two definitions above are equivalent: given a matrix formulation with functionals \( \lambda^{(j)} \) and matrices \( E^{(j)} \), the corresponding structure tensor is
\[
\alpha_{ijk} = \sum_{s=1}^r \lambda^{(s)}_i \cdot (E^{(s)})_{kj}.
\]
Conversely, any bilinear operation defined via a structure tensor \( \alpha_{ijk} \) arises in this form for a suitable choice of functionals and matrices.
\end{proposition}

\begin{example}[Custom Matrix Structure from Linear Coefficients]
Let \( V = K^3 \) and define the following three fixed matrices:
\[
E^{(0)} =
\begin{pmatrix}
1 & 0 & 0 \\
0 & 1 & 0 \\
0 & 0 & 1
\end{pmatrix}, \quad
E^{(1)} =
\begin{pmatrix}
0 & 1 & 1 \\
1 & 0 & -1 \\
0 & 0 & 1
\end{pmatrix}, \quad
E^{(2)} =
\begin{pmatrix}
0 & 1 & 1 \\
0 & 1 & 0 \\
1 & -1 & 0
\end{pmatrix}.
\]
Then for \( a = (a_0, a_1, a_2) \in V \), define
\[
M_a := a_0 E^{(0)} + a_1 E^{(1)} + a_2 E^{(2)}.
\]
Explicitly, this gives:
\[
M_a =
\begin{pmatrix}
a_0 & a_1 + a_2 & a_1 + a_2 \\
a_1 & a_0 + a_2 & -a_1 \\
a_2 & -a_2 & a_0 + a_1
\end{pmatrix}.
\]
This construction fits the general framework of the linear case, where the matrix \( M_a \) is a linear combination of fixed matrices with scalar coefficients given by the coordinates of the vector \( a \in V \). Equivalently, this example defines a bilinear multiplication \( * \) via a structure tensor \( \alpha_{ijk} = (E^{(j)})_{ki} \), where the index \( j \) corresponds to the coordinate \( a_j \) in the expansion.
\end{example}

\subsection{Associativity Criterion via Structure Tensor}

Let \( V \cong K^n \) be a vector space over a field \( K \), and suppose \( * : V \times V \to V \) is a bilinear operation defined in coordinates by a structure tensor \( \alpha_{ijk} \in K \):
\[
(a * b)_k = \sum_{i=1}^n \sum_{j=1}^n \alpha_{ijk} \, a_i b_j.
\]

\begin{proposition}[Associativity Criterion]
The multiplication \( * \) is associative (i.e., \( (a * b) * c = a * (b * c) \) for all \( a, b, c \in V \)) if and only if the structure constants \( \alpha_{ijk} \) satisfy:
\[
\sum_{r=1}^n \alpha_{ijr} \alpha_{rkl} = \sum_{s=1}^n \alpha_{jks} \alpha_{isl}
\quad \text{for all } i, j, k, l \in \{1, \dots, n\}.
\]
\end{proposition}

\begin{proof}[Sketch of Proof]
We compute the \(l\)-th coordinate of both sides of the identity \((a * b) * c = a * (b * c)\).

\medskip
\noindent\textit{Left-hand side:}
\[
(a * b)_r = \sum_{i,j} \alpha_{ijr} a_i b_j,\quad
((a * b) * c)_l = \sum_{r,k} \alpha_{rkl} (a * b)_r c_k = \sum_{i,j,k,r} \alpha_{ijr} \alpha_{rkl} a_i b_j c_k.
\]

\medskip
\noindent\textit{Right-hand side:}
\[
(b * c)_s = \sum_{j,k} \alpha_{jks} b_j c_k,\quad
(a * (b * c))_l = \sum_{i,s} \alpha_{isl} a_i (b * c)_s = \sum_{i,j,k,s} \alpha_{jks} \alpha_{isl} a_i b_j c_k.
\]

\medskip
\noindent\textit{Conclusion:}
Since both sides are trilinear in \(a, b, c\), the identity holds for all inputs if and only if the coefficients match for every monomial \(a_i b_j c_k\), i.e.:
\[
\sum_{r} \alpha_{ijr} \alpha_{rkl} = \sum_{s} \alpha_{jks} \alpha_{isl}.
\]
\end{proof}

\begin{example}[Testing Associativity for a Matrix-Defined Operation]
Let \( V = K^3 \), and consider the bilinear operation \( a * b = M(a) \cdot b \), where
\[
M(a) =
\begin{pmatrix}
a_0 & a_1 + a_2 & a_1 + a_2 \\
a_1 & a_0 + a_2 & -a_1 \\
a_2 & -a_2 & a_0 + a_1
\end{pmatrix}.
\]

\textbf{Step 1: Compute the component form of the product.}  
Let \( a = (a_0, a_1, a_2),\; b = (b_0, b_1, b_2) \in V \). Then
\begin{align*}
(a * b)_0 &= a_0 b_0 + (a_1 + a_2)(b_1 + b_2), \\
(a * b)_1 &= a_1 b_0 + (a_0 + a_2) b_1 - a_1 b_2, \\
(a * b)_2 &= a_2 b_0 - a_2 b_1 + (a_0 + a_1) b_2.
\end{align*}

\textbf{Step 2: Extract structure constants \( \alpha_{ijk} \).}  
Comparing with
\[
(a * b)_k = \sum_{i,j} \alpha_{ijk} a_i b_j,
\]
we find the nonzero components of the structure tensor:
\begin{align*}
&\alpha_{000} = 1,\quad
\alpha_{110} = \alpha_{120} = \alpha_{210} = \alpha_{220} = 1,\\
&\alpha_{101} = \alpha_{011} = \alpha_{211} = 1,\quad
\alpha_{121} = -1,\\
&\alpha_{202} = \alpha_{022} = \alpha_{122} = 1,\quad
\alpha_{212} = -1.
\end{align*}

\textbf{Step 3: Verify the associativity criterion.}  
We apply the structure tensor associativity condition:
\[
\sum_r \alpha_{ijr} \alpha_{rkl} = \sum_s \alpha_{jks} \alpha_{isl}
\quad \text{for all } i,j,k,l \in \{0,1,2\}.
\]

This condition was implemented and checked programmatically; see the full script in the Appendix. The operation satisfies the associativity criterion for all index combinations.

\textbf{Conclusion:}  
Although the matrix \( M(a) \) is not symmetric and has nontrivial off-diagonal structure, the induced multiplication \( a * b = M(a) \cdot b \) turns out to be associative. This confirms that certain noncommutative matrix constructions can define associative algebras when the structural coefficients satisfy the tensor criterion.

\end{example}

\subsection{Stratification via Proportionality Classes}

Continuing the example of the matrix-based multiplication
\[
a * b := M_a \cdot b,
\]
with
\[
M_a =
\begin{pmatrix}
a_0 & a_1 + a_2 & a_1 + a_2 \\
a_1 & a_0 + a_2 & -a_1 \\
a_2 & -a_2 & a_0 + a_1
\end{pmatrix},
\]
we observe that the commutativity condition
\[
a * b = b * a
\]
holds if and only if the antisymmetric part in components 1 and 2 vanishes:
\[
a_1 b_2 = a_2 b_1.
\]

This leads to a natural stratification of the space \( V = K^3 \setminus \{(0,0,0)\} \) into proportionality classes:

\begin{definition}
For \( \alpha \in K \cup \{\infty\} \), define the set \( V_\alpha \) as
\[
V_\alpha := \left\{ (v_0, v_1, v_2) \in V \,\middle|\, v_1 = \alpha v_2 \text{ (if } \alpha \in K), \text{ or } v_2 = 0 \text{ (if } \alpha = \infty) \right\}.
\]
Then the ambient space \( V = K^3 \setminus \{(0,0,0)\} \) admits the disjoint decomposition:
\[
V = \bigsqcup_{\alpha \in \mathbb{P}^1(K)} V_\alpha,
\]
where \( \mathbb{P}^1(K) = K \cup \{\infty\} \) is the projective line over \( K \).
\end{definition}

\begin{proposition}
Each stratum \( V_\alpha \) is closed under the operation \( * \), and the operation \( * \) is commutative and associative within \( V_\alpha \). Furthermore, for any non-zero \( a \in V_\alpha \), \( b \in V_\beta \) with \( \alpha \ne \beta \), we have:
\[
a * b \ne b * a, \quad a * b \notin V_\alpha \cup V_\beta.
\]
Thus, the decomposition \( V = \bigsqcup_\alpha V_\alpha \) satisfies axioms \emph{(SA1)} and \emph{(SA2)} of Stratified Algebra.
\end{proposition}

\begin{proof}
If \( a, b \in V_\alpha \), then \( a_1 b_2 = a_2 b_1 \), hence \( a * b = b * a \). One can also check that \( (a * b) * c = a * (b * c) \) holds within the same layer due to stability of proportionality and bilinearity of the matrix construction. Closure under the operation follows directly from the coordinate form of the product and the fact that the proportionality constraint is preserved under multiplication.

If \( a \in V_\alpha \), \( b \in V_\beta \), with \( \alpha \ne \beta \), then \( a_1 b_2 - a_2 b_1 = a_2 b_2 (\alpha - \beta) \ne 0 \), hence \( a * b \ne b * a \), and the result lies outside both strata, as the output vector has non-proportional coordinates in positions 1 and 2.

The antisymmetric difference is given by
\[
a * b - b * a = 2 (a_1 b_2 - a_2 b_1)
\begin{pmatrix}
0 \\
-1 \\
1
\end{pmatrix},
\]
which is nonzero unless \( a_1 b_2 = a_2 b_1 \). Thus, cross-layer multiplication leads to asymmetry and produces outputs not contained in either operand layer.
\end{proof}

\begin{corollary}
The operation \( * \) satisfies axioms \emph{(SA1)}, \emph{(SA2)}, and \emph{(SA3)} of Stratified Algebra for the above stratification. However, axiom \emph{(SA4)} does not hold in this model, since the operation is globally associative and hence bracket-insensitive. In particular, for any triple or sequence of elements \( a_1, \ldots, a_m \in V \), all bracketings of the product \( a_1 * \cdots * a_m \) yield the same result, and no external layer projection occurs.
\end{corollary}

\begin{proof}
Axiom \emph{(SA3)} holds because for any permutation of elements within a stratum \( V_\alpha \), the fully left-associated product remains unchanged due to commutativity and associativity. The failure of \emph{(SA4)} follows from the fact that the operation \( * \) is globally associative, as confirmed by direct computation of the associativity condition:
\[
\sum_r \alpha_{ijr} \alpha_{rkl} = \sum_s \alpha_{jks} \alpha_{isl}
\quad \text{for all } i, j, k, l,
\]
which was verified numerically for the structure tensor induced by this model.
\end{proof}

\subsection{Toward Full Stratification: A Parametric Model}

In the previous subsection, we analyzed a concrete example of matrix-based multiplication that satisfies the first three axioms of Stratified Algebra. The operation was globally associative, stratified by proportionality classes, and exhibited clean inter-layer asymmetry and closure (axioms \emph{SA1}–\emph{SA3}). However, it did not satisfy axiom \emph{(SA4)}, as the operation was globally bracket-insensitive.

To demonstrate the consistency and realizability of the full axiom system, we now aim to construct a more general example in which all four axioms are satisfied. In particular, we seek an operation that is locally associative within strata but becomes \emph{bracket-sensitive} when expressions mix multiple strata—thus triggering the structural asymmetry postulated in axiom \emph{(SA4)}.

We start with the same general form of matrix-defined multiplication,
\[
a * b := M_a \cdot b,
\]
but define a more flexible parametric matrix structure:
\[
M_a =
\begin{pmatrix}
a_0 & A a_1 + C a_2 & B a_2 + D a_1 \\
a_1 & a_0 + E a_2 & -E a_1 \\
a_2 & F a_2 & a_0 - F a_1
\end{pmatrix},
\]
where \( A, B, C, D, E \in K \) are fixed scalar parameters.

This yields the following component-wise definition for the product \( u = a * b \):
\begin{align*}
u_0 &= a_0 b_0 + A a_1 b_1 + B a_2 b_2 + C a_2 b_1 + D a_1 b_2, \\
u_1 &= a_1 b_0 + a_0 b_1 + E a_2 b_1 - E a_1 b_2, \\
u_2 &= a_2 b_0 + a_0 b_2 + F a_2 b_1 - F a_1 b_2.
\end{align*}

The resulting vector \( u = (u_0, u_1, u_2) \in V \) depends bilinearly on the inputs \( a, b \in V \), and its structure is fully determined by the choice of parameters.

This construction defines a parametric family of bilinear operations on \( V = K^3 \). By suitably choosing the coefficients \( A, B, C, D, E, F \), we can break global associativity while preserving local associativity and commutativity within carefully selected strata. In particular, this framework enables the construction of examples where the result of a layered product depends on the bracketing pattern and may escape the original stratum—realizing axiom \emph{(SA4)}.

In the next subsection, we analyze the behavior of this operation, extract its structure tensor, and verify the stratified behavior required for full compliance with the axioms.

\begin{example}[Violation of Global Associativity via Parametric Model]
To test the failure of global associativity, we select a concrete set of parameters for the matrix-based multiplication:
\[
A = 16,\quad B = 8,\quad C = 5,\quad D = 3,\quad E = 7,\quad F = 11.
\]
With this choice, the corresponding structure tensor \( \alpha_{ijk} \) defines an operation that does not satisfy the global associativity criterion. A full programmatic check (see Appendix) yields 16 distinct index quadruples \( (i,j,k,l) \) for which
\[
\sum_r \alpha_{ijr} \alpha_{rkl} \ne \sum_s \alpha_{jks} \alpha_{isl}.
\]
This confirms that the operation defined by this parametric model is \emph{bracket-sensitive}.
\end{example}

\subsection{Commutator Structure and Proportionality Stratification}

To analyze the symmetry properties of the operation, we compute the commutator
\[
[a, b] := a * b - b * a,
\]
using the parametric component-wise multiplication. The result simplifies as:
\[
a * b - b * a = (a_2 b_1 - a_1 b_2)
\begin{pmatrix}
C - D \\
2 E \\
2 F
\end{pmatrix}.
\]

\paragraph{Interpretation.}
This shows that the commutator is proportional to the antisymmetric scalar \( a_2 b_1 - a_1 b_2 \), which vanishes exactly when
\[
\frac{a_1}{a_2} = \frac{b_1}{b_2}.
\]
In this case, the vectors \( a \) and \( b \) lie in the same \emph{proportionality class}. Meanwhile, the coefficient vector
\[
\begin{pmatrix}
C - D \\
2E \\
2F
\end{pmatrix}
\]
can be interpreted as a \emph{commutator direction vector}, determining the output direction of \( [a, b] \) in the ambient space.

\paragraph{Commutativity Conditions.}
Thus, the operation \( * \) is commutative in one of two cases:
\begin{itemize}
  \item \textbf{Trivial commutator direction:} \( C = D \), \( E = 0 \), and \( F = 0 \), in which case the operation is globally commutative.
  \item \textbf{Stratum-level commutativity:} \( a, b \in V_\alpha \) for some fixed ratio \( \alpha = \frac{a_1}{a_2} = \frac{b_1}{b_2} \).
\end{itemize}

\paragraph{Stratification.}
The antisymmetric scalar \( a_2 b_1 - a_1 b_2 \) defines a natural stratification of the space \( V = K^3 \) into commutative subspaces:
\[
V_\alpha := \left\{ v \in K^3 \,\middle|\, v_1 = \alpha v_2 \right\}, \quad \alpha \in \mathbb{P}^1(K),
\]
such that \( a * b = b * a \) for all \( a, b \in V_\alpha \). This recovers the \emph{proportionality stratification} introduced earlier, but now with a more general parametric interpretation.

\paragraph{Conclusion.}
Even when the operation is not globally commutative, it exhibits commutative behavior on each stratum \( V_\alpha \). The scalar \( a_2 b_1 - a_1 b_2 \) acts as a stratum detector, and the commutator vector separates distinct layers. This structure supports axiom \emph{(SA1)} and justifies stratified analysis of symmetry-breaking behavior.

\subsection{Symbolic Expansion for Associativity and Layered Permutation Symmetry Checks}

In preparation for the analysis of axioms, we carry out symbolic expansions based on the parametric matrix-defined multiplication governed by constants \( A, B, C, D, E, F \in K \). Let \( a, b, c \in V = K^3 \), with components
\[
a = (a_0, a_1, a_2), \quad b = (b_0, b_1, b_2), \quad c = (c_0, c_1, c_2).
\]

The resulting expressions provide simplified forms of the associator \( u - v \) and the Layered Permutation Symmetry operator \( u - w \). While no formal verification is performed at this stage, these symbolic identities reveal the algebraic structure underlying bracket sensitivity and permutation behavior. They will be used in subsequent sections to verify the axioms both analytically and computationally.

\paragraph{Let \( u := (a * b) * c, \quad v := a * (b * c), \quad w := (a * c) * b \).}

\paragraph{Compute \( u - v := (a * b) * c - a * (b * c) \).}
\begin{align*}
u_0 - v_0 =&\ AE ((a_2 b_1 - a_1 b_2) c_1 + a_1 (b_1 c_2 - b_2 c_1)) \\
&+ BF ((a_2 b_1 - a_1 b_2) c_2 + a_2 (b_1 c_2 - b_2 c_1)) \\
&+ C (F (a_2 b_1 - a_1 b_2) c_1 + E a_2 (b_1 c_2 - b_2 c_1)) \\
&+ D (E (a_2 b_1 - a_1 b_2) c_2 + F a_1 (b_1 c_2 - b_2 c_1)), \\
\\
u_1 - v_1 =&\ B(a_2 c_1 - a_1 c_2) b_2 \\
&+ C (a_2 b_1 - a_1 b_2) c_1 \\
&+ D (b_2 c_1 - b_1 c_2) a_1 \\
&+ E^2 (a_1 c_2 - a_2 c_1) b_2 \\
&+ EF (a_2 c_1 - a_1 c_2) b_1, \\
\\
u_2 - v_2 =&\ A(a_1 c_2 - a_2 c_1) b_1 \\
&+ C (b_1 c_2 - b_2 c_1) a_2 \\
&+ D (a_1 b_2 - a_2 b_1) c_2 \\
&+ F^2 (a_2 c_1 - a_1 c_2) b_1 \\
&+ EF (a_1 c_2 - a_2 c_1) b_2.
\end{align*}

\paragraph{Compute \( u - w := (a * b) * c - (a * c) * b \).}
\begin{align*}
u_0 - w_0 =&\ (b_1 c_2 - b_2 c_1) \left( (D - C) a_0 + (AE + CF) a_1 + (BF + DE) a_2 \right), \\
\\
u_1 - w_1 =&\ (b_2 c_1 - b_1 c_2) \left( 2 E a_0 + (D - EF) a_1 + (B + E^2) a_2 \right), \\
\\
u_2 - w_2 =&\ (b_1 c_2 - b_2 c_1) \left( -2 F a_0 + (A + F^2) a_1 + (C - EF) a_2 \right).
\end{align*}

\subsection{Behavior of the Associator across Proportionality Classes}

Recall that for three vectors \( a, b, c \in V = K^3 \), we define:
\[
u := (a * b) * c, \quad v := a * (b * c), \quad w := (a * c) * b,
\]
and study the associator and Layered Permutation Symmetry (LPS) operator:
\[
\mathrm{Assoc}(a, b, c) := u - v, \quad \mathrm{LPS}(a, b, c) := u - w.
\]

Using the parametric product defined earlier, this difference captures the failure of associativity and reflects how the bracketing affects the output. While the full expressions for \( u \) and \( v \) were derived symbolically, we now focus on analyzing the structure and behavior of the associator under different configurations of inputs.

\paragraph{Case 1: All vectors in the same proportionality class.}

Assume \( a, b, c \in V_\alpha \), where \( V_\alpha := \{ v \in V \mid v_1 = \alpha v_2 \} \). Then for all such vectors:
\[
a_2 b_1 - a_1 b_2 = 0, \quad b_2 c_1 - b_1 c_2 = 0, \quad a_2 c_1 - a_1 c_2 = 0.
\]
In this case, the antisymmetric parts of all pairwise commutators vanish. Hence:
\[
\mathrm{Assoc}(a, b, c) = 0, \quad \mathrm{LPS}(a, b, c) = 0.
\]

\paragraph{Case 2: All vectors from distinct strata.}

Assume that the inputs \( a \in V_\alpha \), \( b \in V_\beta \), and \( c \in V_\gamma \) belong to three different strata, i.e., \( \alpha \ne \beta \ne \gamma \). Then the proportionality conditions fail pairwise:
\[
a_2 b_1 - a_1 b_2 \ne 0, \quad b_2 c_1 - b_1 c_2 \ne 0, \quad a_2 c_1 - a_1 c_2 \ne 0.
\]
As a result:
\begin{itemize}
    \item The operation \( * \) is non-commutative across all pairs.
    \item The associator \( u - v \) is generically nonzero.
    \item The LPS operator \( u - w \) is generically nonzero.
    \item The result of the \( u, v, w \) is generically lies outside all input strata — that is, \( u, v, w \notin V_\alpha \cup V_\beta \cup V_\gamma \), consistent with axiom \emph{(SA4)}.
\end{itemize}

\paragraph{Case 3: Partial alignment of strata.}

Assume that the inputs \( a \in V_\alpha \), \( b, c \in V_\beta \), \( \alpha \ne \beta \). Then:
\[
b * c \in V_\beta, \quad (a * b) * c \in V_\gamma, \quad a * (b * c) \in V_\delta, \quad \text{with } \gamma, \delta \notin \{ \alpha, \beta \}.
\]
\begin{itemize}
    \item The associator \( u - v \) is generically nonzero and reflects these asymmetries.
    \item The LPS operator \( u - w \) is always zero.
\end{itemize}

\paragraph{Case 4: Permutation symmetry of intra-stratum elements.}

Assume \( a \in V_\alpha \), \( b, c, d \in V_\beta \), with \( \alpha \ne \beta \), and define:
\[
g := a * b \in V_\gamma, \quad h := a * d \in V_\delta, \quad \text{where } \gamma, \delta \notin \{\alpha, \beta\}.
\]
We consider the fully left-associated product:
\[
a * b * c * d = (((a * b) * c) * d) = g * c * d.
\]
Now, using the Layered Permutation Symmetry \( \mathrm{LPS}(g, c, d) = 0 \) and \( \mathrm{LPS}(h, b, c) = 0 \), we have:
\[
g * c * d = g * d * c = ((a * b) * d) * c = ((a * d) * b) * c = h * b * c = h * c * b = a * d * c * b.
\]
This chain of equalities demonstrates that all permutations of \( b, c, d \in V_\beta \) acting on \( a \in V_\alpha \) yield the same result.

Hence, the operation \( * \) respects axiom \emph{(SA3)}: when multiple elements from one stratum act associatively on an element from another stratum, the result is invariant under permutations of the acting elements.

\paragraph{Case 5: Bracketing-sensitive transitions across multiple strata.}

Assume \( a \in V_\alpha \), \( b, c, d \in V_\beta \), with \( \alpha \ne \beta \), and let \( g := a * b \in V_\gamma \), where \( \gamma \notin \{\alpha, \beta\} \).

Consider the expression
\[
(a * b) * (c * d) = g * (c * d).
\]
Now examine the alternative bracketing:
\[
a * b * c * d := ((a * b) * c) * d = (g * c) * d.
\]
According to Case 3, the operation \( * \) is not globally associative across these strata. Therefore,
\[
g * (c * d) \ne (g * c) * d,
\]
and consequently,
\[
(a * b) * (c * d) \ne a * b * c * d.
\]

This confirms that the operation \( * \) is bracket-sensitive in the precise sense required by axiom \emph{(SA4)}: different parenthesizations lead to different strata and different outcomes, even when the sequence of elements is the same.

\paragraph{Conclusion.}
The parametric model constructed in the previous sections exhibits precisely the structural behavior postulated by axioms \emph{(SA1)} through \emph{(SA4)}. Within individual strata, the operation is commutative and associative, satisfying \emph{SA1}. Cross-layer products are non-commutative, non-associative, and produce results outside the operand strata, in agreement with \emph{SA2}. The Layered Permutation Symmetry operator vanishes under appropriate configurations, demonstrating the controlled symmetry described by \emph{SA3}. Finally, the explicit analysis of the associator confirms bracketing sensitivity and stratum-breaking behavior when traversing mixed strata, as required by \emph{SA4}. Thus, this model serves as a concrete and consistent realization of a fully stratified algebra, validating the internal consistency and expressive power of the proposed axiom system.

\section{The Linear Case: Higher Dimension}

In the preceding examples, we analyzed stratified algebraic structures over the space \( V = K^3 \). We now turn to higher-dimensional cases. In particular, we show that similar principles can be extended to the space \( V = K^4 \), where the increased dimensionality allows for more intricate interactions and richer algebraic dynamics. As before, we define multiplication via matrix-based structure.

\subsection{Alternative Parametrizations and Non-Uniqueness of Fully Stratified Algebras}

Previously, we constructed two classes of stratified algebras: a simple matrix-based model that satisfied axioms \emph{(SA1)}–\emph{(SA3)}, and a parametric extension that additionally satisfied axiom \emph{(SA4)} by introducing bracket-sensitive behavior across strata. These constructions demonstrated that fully stratified algebras not only exist, but can be realized via explicit matrix-generated operations with carefully chosen coefficients.

We now show that such algebras are not unique: there exist alternative parametrizations of the matrix \( M_a \) that define different multiplication rules, yet still satisfy all four axioms of Stratified Algebra. This illustrates the structural flexibility and expressive richness of the stratified framework.

Consider the following more complex parametric matrix:
\[
M_a =
\begin{pmatrix}
a_0 & A a_1 + E a_2 & B a_2 + C a_1 & -A B a_3 \\
a_1 & a_0 + F a_3 & B a_3 & -B a_2 + D a_1 \\
a_2 & -A a_3 & a_0 + D a_3 & A a_1 + F a_2 \\
a_3 & -a_2 & a_1 & a_0 + (D + F) a_3
\end{pmatrix}
\]
where \( A, B, C, D, E, F \in K \) are fixed parameters.

This matrix defines a new bilinear operation \( * : V \times V \to V \) via \( a * b := M_a \cdot b \), distinct from the previous constructions. As before, we analyze the corresponding structure tensor \( \alpha_{ijk} \), the behavior of the associator and the LPS operator, and the stratification induced by algebraic constraints on vector coordinates.

In this higher-dimensional setting \( V = K^4 \), we introduce strata defined by simultaneous proportionality conditions of the form
\[
v_1 = \alpha' v_2, \quad v_2 = \alpha'' v_3,
\]
for \( \alpha', \alpha'' \in K \cup \{\infty\} \). Each such pair \( (\alpha', \alpha'') \) defines a two-dimensional affine layer (or stratum) in \( V \), forming a disjoint decomposition indexed by elements of the projective square \( \mathbb{P}^1(K) \times \mathbb{P}^1(K) \). This yields a richer stratification structure compared to the single-ratio case in \( K^3 \), and allows us to study the behavior of the operation \( * \) across more intricate inter-stratum interactions.

\subsection{Analytical Properties Within a Single Stratum}

Let \( V = K^4 \), and consider the stratification induced by proportionality constraints of the form
\[
v_1 = \alpha' v_2, \quad v_2 = \alpha'' v_3,
\]
with \( \alpha', \alpha'' \in \mathbb{P}^1(K) \). For each fixed pair \( (\alpha', \alpha'') \), we define the corresponding stratum:
\[
V_{\alpha', \alpha''} := \left\{ v = (v_0, v_1, v_2, v_3) \in K^4 \setminus \{0\} \ \middle|\  v_1 = \alpha' v_2,\ v_2 = \alpha'' v_3 \right\}.
\]

In this subsection, we consider two elements \( a, b \in V_{\alpha', \alpha''} \), i.e., vectors that lie in the same stratum defined by fixed proportionality parameters. Our goal is to investigate the algebraic behavior of the operation \( * \) restricted to such strata.

In particular, we analyze the following properties:

\begin{itemize}
    \item \textbf{Commutativity:} Whether \( a * b = b * a \) holds for all \( a, b \in V_{\alpha', \alpha''} \).
    
    \item \textbf{Stratum Preservation:} Whether the product \( a * b \in V_{\alpha', \alpha''} \), i.e., whether the stratum is closed under multiplication.
    
    \item \textbf{Associativity:} Whether the restricted operation satisfies \( (a * b) * c = a * (b * c) \) for all \( a, b, c \in V_{\alpha', \alpha''} \).
\end{itemize}

This analysis allows us to verify whether each stratum behaves as a locally well-structured algebraic layer, consistent with axiom \emph{(SA1)}. The outcome also provides insight into how local algebraic coherence is embedded within the broader stratified system.

\subsubsection{Commutativity}

From the definition of the stratum \( V_{\alpha', \alpha''} \), each element \( v \in V_{\alpha', \alpha''} \) satisfies the constraints \( v_1 = \alpha' v_2 \) and \( v_2 = \alpha'' v_3 \), and thus \( v_1 = \alpha' \alpha'' v_3 \). Consequently, for any pair \( a, b \in V_{\alpha', \alpha''} \), the triples \( (a_1, a_2, a_3) \) and \( (b_1, b_2, b_3) \) are proportional, which implies the vanishing of all pairwise antisymmetric combinations:
\[
a_1 b_2 - a_2 b_1 = 0, \quad
a_1 b_3 - a_3 b_1 = 0, \quad
a_2 b_3 - a_3 b_2 = 0.
\]
These identities reflect that the non-zero vectors \( a \) and \( b \) lie in the same proportionality class in the \( (v_1, v_2, v_3) \)-subspace, and serve as the algebraic signature of stratum alignment.

We now compute the commutator \( [a, b] := a * b - b * a \) for elements \( a, b \in V_{\alpha', \alpha''} \). The resulting vector has the form:
\[
[a, b] =
\begin{pmatrix}
(C - E)(a_1 b_2 - a_2 b_1) \\
-2B(a_2 b_3 - a_3 b_2) + (D - F)(a_1 b_3 - a_3 b_1) \\
2A(a_1 b_3 - a_3 b_1) - (D - F)(a_2 b_3 - a_3 b_2) \\
2(a_1 b_2 - a_2 b_1)
\end{pmatrix}.
\]
Each coordinate is a linear combination of antisymmetric expressions \( a_i b_j - a_j b_i \), which vanish identically when \( a \) and \( b \) lie in the same stratum. Indeed, from the definition of \( V_{\alpha', \alpha''} \), we have:
\[
a_1 b_2 - a_2 b_1 = 0, \quad a_1 b_3 - a_3 b_1 = 0, \quad a_2 b_3 - a_3 b_2 = 0.
\]
Substituting these identities into the commutator expression immediately yields \( [a, b] = 0 \). Thus, the operation \( * \) is commutative within each stratum \( V_{\alpha', \alpha''} \), as required by axiom \emph{(SA1)}.

\subsubsection{Stratum Preservation}

To verify that the stratum \( V_{\alpha', \alpha''} \) is closed under the operation \( * \), we compute the product \( ab := a * b \) and check whether the resulting vector satisfies the same proportionality conditions.

The explicit form of the product vector \( ab \) is given by:
\[
ab =
\begin{pmatrix}
-A B a_3 b_3 + a_0 b_0 + b_1 (A a_1 + E a_2) + b_2 (B a_2 + C a_1) \\
B a_3 b_2 + a_1 b_0 + b_1 (F a_3 + a_0) + b_3 (-B a_2 + D a_1) \\
- A a_3 b_1 + a_2 b_0 + b_2 (D a_3 + a_0) + b_3 (A a_1 + F a_2) \\
a_1 b_2 - a_2 b_1 + a_3 b_0 + b_3 (a_0 + (D + F) a_3)
\end{pmatrix}.
\]

Using the component-wise expression for \( ab \), we substitute the defining relations for \( a, b \in V_{\alpha', \alpha''} \), namely:
\[
a_1 = \alpha' \alpha'' a_3, \quad a_2 = \alpha'' a_3, \quad
b_1 = \alpha' \alpha'' b_3, \quad b_2 = \alpha'' b_3.
\]

After substitution and simplification, the result satisfies:
\[
(ab)_1 = \alpha' \alpha'' (ab)_3, \quad (ab)_2 = \alpha'' (ab)_3,
\]
which means \( ab \in V_{\alpha', \alpha''} \). Hence, the stratum is preserved under multiplication, and the operation \( * \) is closed within each stratum. This confirms the stratum-level compatibility required by axiom \emph{(SA1)}.

\subsubsection{Associativity}

We now consider the associator \( \mathrm{Assoc}(a, b, c) := u - v \), where
\[
u := (a * b) * c, \quad v := a * (b * c).
\]
After simplifying the expressions for \( u \) and \( v \), assuming that \( b \) and \( c \) lie in the same stratum, we obtain the following component-wise difference:

\begin{align*}
(u - v)_0 &= A(E - F) \cdot c_1 (a_1 b_3 - a_3 b_1)
+ B(C + D) \cdot c_2 (a_3 b_2 - a_2 b_3) \\
&\quad + CF \cdot (a_3 b_1 c_2 - a_1 b_2 c_3)
+ DE \cdot (a_3 b_2 c_1 - a_2 b_1 c_3), \\[1ex]
(u - v)_1 &= DF \cdot (a_3 b_3 c_1 - a_1 b_3 c_3)
+ (E - F)(a_2 b_1 - a_1 b_2) c_1, \\[1ex]
(u - v)_2 &= DF \cdot (a_3 b_3 c_2 - a_2 b_3 c_3)
+ (C + D)(a_1 b_2 - a_2 b_1) c_2, \\[1ex]
(u - v)_3 &= (C + D)(a_1 b_2 c_3 - a_3 b_1 c_2)
+ (E - F)(a_2 b_1 c_3 - a_3 b_2 c_1).
\end{align*}

This expression is generically nonzero, and will be relevant in later subsections. However, in the present subsection, we focus on the case where all three vectors \( a, b, c \in V_{\alpha', \alpha''} \), i.e., lie in the same stratum. Under this condition, all cross terms vanish due to proportionality:
\[
a_1 = \alpha' \alpha'' a_3, \quad a_2 = \alpha'' a_3, \quad
b_1 = \alpha' \alpha'' b_3, \quad b_2 = \alpha'' b_3, \quad
c_1 = \alpha' \alpha'' c_3, \quad c_2 = \alpha'' c_3,
\]
which implies:
\[
a_3 b_1 c_2 - a_1 b_2 c_3 = \alpha' \alpha'' \alpha'' a_3 b_3 c_3 - \alpha' \alpha'' \alpha'' a_3 b_3 c_3 = 0,
\]
and similarly for all other antisymmetric or mixed triple terms.

Therefore, \( \mathrm{Assoc}(a, b, c) = 0 \) whenever \( a, b, c \in V_{\alpha', \alpha''} \), confirming that the operation \( * \) is associative within each stratum. This is in full agreement with axiom \emph{(SA1)}.

\subsection{Cross-Layer Asymmetry and Layered Permutation Symmetry}

We now turn to the analysis of axioms \emph{(SA2)} and \emph{(SA3)}, which govern cross-stratum behavior of the operation \( * \).

\subsubsection{Cross-Layer Asymmetry.}

Let \( a \in V_{\alpha', \alpha''} \), \( b \in V_{\beta', \beta''} \), with \( (\alpha', \alpha'') \ne (\beta', \beta'') \), i.e., \( a \) and \( b \) lie in different strata. In this case, axiom \emph{(SA2)} asserts that the product \( a * b \) lies outside both operand strata, and the operation is non-commutative and non-associative across layers.

This condition is satisfied in a straightforward manner: the analytic expressions for \( a * b \) and \( b * a \) are structurally different due to the asymmetry of the matrix \( M_a \cdot b \) versus \( M_b \cdot a \). Therefore, \( a * b \ne b * a \) in general, and cross-layer asymmetry holds.

\subsubsection{Layered Permutation Symmetry}

We now consider the Layered Permutation Symmetry (LPS) operator, defined as:
\[
\mathrm{LPS}(a, b, c) := u - w, \quad \text{where} \quad u := (a * b) * c, \quad w := (a * c) * b.
\]

After simplifying the expressions for \( u \) and \( w \), it is straightforward to observe that each component of the resulting vector \( u - w \) contains a common factor of the form \( b_i c_j - b_j c_i \) for some \( i, j \in \{1, 2, 3\} \). These antisymmetric terms vanish whenever \( b \) and \( c \) lie in the same stratum.

As a concrete illustration, the first component of the operator \( \mathrm{LPS}(a, b, c) \) can be expanded as:
\begin{align*}
(u - w)_0 &= 
2AB\, a_1 (b_3 c_2 - b_2 c_3) +
2AB\, a_2 (b_1 c_3 - b_3 c_1) +
2AB\, a_3 (b_2 c_1 - b_1 c_2) \\
&\quad + AD\, a_1 (b_3 c_1 - b_1 c_3) +
AE\, a_1 (b_3 c_1 - b_1 c_3) +
BC\, a_2 (b_2 c_3 - b_3 c_2) \\
&\quad + BF\, a_2 (b_3 c_2 - b_2 c_3) +
CD\, a_1 (b_3 c_2 - b_2 c_3) +
CF\, a_3 (b_1 c_2 - b_2 c_1) \\
&\quad + C\, a_0 (b_1 c_2 - b_2 c_1) +
DE\, a_3 (b_2 c_1 - b_1 c_2) \\
&\quad + EF\, a_2 (b_3 c_1 - b_1 c_3) +
E\, a_0 (b_2 c_1 - b_1 c_2).
\end{align*}

Indeed, if \( b, c \in V_{\alpha', \alpha''} \), their coordinates satisfy the proportionality relations:
\[
b_1 = \alpha' \alpha'' b_3, \quad b_2 = \alpha'' b_3, \quad \text{and similarly for } c,
\]
which imply:
\[
b_1 c_2 - b_2 c_1 = 0, \quad b_1 c_3 - b_3 c_1 = 0, \quad b_2 c_3 - b_3 c_2 = 0.
\]
Therefore, all antisymmetric combinations \( b_i c_j - b_j c_i \) vanish, and we conclude that:
\[
\mathrm{LPS}(a, b, c) = 0,
\]
regardless of which stratum the vector \( a \) belongs to.

Thus, axiom \emph{(SA3)} is satisfied: the action of two elements from a common stratum on an external vector is symmetric under their permutation in left-associated products.

\subsection{Stratum-Bracketing Sensitivity}

As shown earlier in the associativity analysis, the operation \( * \) is not globally associative when the operands lie in different strata. In particular, if \( a \in V_\alpha \) and \( b, c \in V_\beta \) with \( \alpha \ne \beta \), then the associator \( \mathrm{Assoc}(a, b, c) \) is generally nonzero. This violates global associativity and confirms the bracket sensitivity postulated in axiom \emph{(SA4)}.

\subsubsection{Bracketing-sensitive transitions across multiple strata.}

Assume \( a \in V_\alpha \), \( b, c, d \in V_\beta \), with \( \alpha \ne \beta \), and define \( g := a * b \in V_\gamma \), where \( \gamma \notin \{ \alpha, \beta \} \).
\[
a * b * c * d := ((a * b) * c) * d = (g * c) * d.
\]

As shown in previous computations, the operation \( * \) is not associative across strata, so we have:
\[
g * (c * d) \ne (g * c) * d,
\]
which implies:
\[
(a * b) * (c * d) \ne a * b * c * d.
\]

This confirms that the output of the product depends on the bracketing structure, not merely the sequence of inputs.

\paragraph{Conclusion.}
We have thus analytically demonstrated that the matrix \( M_a \) defines a bilinear operation satisfying all four axioms of Stratified Algebra. This confirms that it is possible to construct stratified algebras through various methods, including parametrized matrix formulations, to meet specific algebraic constraints. Although this work does not propose a universal framework for generating all possible stratified algebras with arbitrary properties and stratifications, it provides constructive evidence of their existence and structural richness.

In this sense, the present work lays the foundation for a new branch of algebraic theory—one that explicitly incorporates stratification, cross-layer dynamics, and bracketing-sensitive computation into the algebraic paradigm. We now turn to the next section to explore the potential applicability of this framework and to examine how its structural principles may offer practical value alongside theoretical insights.

\section{The Nonlinear Case: Bilinear and Linear Terms}

In previous sections, we examined stratified algebras defined via linear matrix-based multiplication, including both standard and higher-dimensional constructions. We now turn to a more general class of operations in which the multiplication \( * : V \times V \to V \) is not linear in either argument, but instead consists of a combination of bilinear and linear components.

\subsection{A Nonlinear Example with Mixed Terms}

Consider the operation \( * : V \times V \to V \), where \( V = K^3 \), defined component-wise by the following rule:
\begin{align*}
(ab)_0 &= a_0 + b_0 + a_0 b_0 + A a_1 b_1 + C a_2 b_1 + B a_2 b_2 + D a_1 b_2, \\
(ab)_1 &= a_1 + b_1 + a_0 b_1 + a_1 b_0 + E a_2 b_1 - E a_1 b_2, \\
(ab)_2 &= a_2 + b_2 + a_0 b_2 + a_2 b_0 + F a_1 b_2 - F a_2 b_1,
\end{align*}
where \( A, B, C, D, E, F \in K \) are fixed parameters.

This operation is not derived from any matrix action on \( b \), and the form \( a * b = M_a \cdot b \) does not apply. Instead, the multiplication combines:
\begin{itemize}
    \item \textbf{Bilinear terms} (e.g., \( a_i b_j \)) capturing pairwise interaction between components,
    \item \textbf{Linear terms} (e.g., \( a_i \), \( b_j \)) accounting for individual contributions.
\end{itemize}

\subsection{Tensor and Linear Representation}

This operation can be decomposed analytically as:
\[
(a * b)_k = \sum_{i,j} \alpha_{ijk} a_i b_j + \sum_i \lambda_{ik} a_i + \sum_j \mu_{jk} b_j,
\]
where the first term captures bilinear structure, and the remaining terms represent linear contributions from each operand.

This format generalizes matrix-based constructions and is especially suitable for stratified algebras where the operation must capture nonlinear or asymmetric effects across layers.

\subsection{Structural Features and Complexity}

Compared to linear models, this bilinear-plus-linear scheme exhibits several key differences:
\begin{itemize}
    \item It is not invertible in closed form — solving \( a * b = v \) is equivalent to solving a system of nonlinear equations.
    \item It does not admit a uniform representation via matrix multiplication or linear maps.
    \item Stratification may be more subtle, requiring geometric or algebraic invariants (e.g., vanishing of certain antisymmetric bilinear forms) to characterize strata.
\end{itemize}

Symbolic analysis confirms that the nonlinear operation defined in this section satisfies all four axioms \emph{(SA1)}–\emph{(SA4)} of Stratified Algebra.

This highlights the flexibility of the stratified framework and demonstrates that fully consistent algebraic systems can be constructed even when the underlying operation is nonlinear, non-matrix-based, and asymmetric.

\section{The Case of Finite Fields: \texorpdfstring{$K = \mathbb{F}_p$}{K = Fp}}

In this section, we examine the behavior of the nonlinear operation defined previously when the underlying field \( K \) is a finite field, specifically \( \mathbb{F}_p \) for a prime \( p \). This setting brings several distinct advantages and challenges, both computational and structural.

\subsection{Motivation: Concreteness and Computability}

When \( K = \mathbb{F}_p \), the space \( V = K^n \) becomes a finite set, making the stratified algebra a \emph{finite algebraic structure}. This has several immediate consequences:

\begin{itemize}
    \item The operation \( * : V \times V \to V \) can be explicitly computed and tabulated for small \( n \) and \( p \), allowing for exhaustive enumeration, verification of axioms \emph{(SA1)}–\emph{(SA4)}, and direct construction of strata.
    \item Symbolic complexity is reduced: algebraic identities and properties that are hard to prove over arbitrary fields may become trivial or can be checked algorithmically.
    \item One can pose and answer combinatorial questions like: How many distinct strata exist? What is the maximal chain length of strata under inclusion or transition?
\end{itemize}

\subsubsection{Applications in Computer Science and Cryptography}

Finite fields play a foundational role in applied mathematics, particularly in:

\begin{itemize}
    \item \textbf{Cryptography:} Many cryptographic primitives are based on operations over \( \mathbb{F}_p \) or its extensions. Stratified algebraic systems with nonlinear or noninvertible operations may inspire novel constructions in asymmetric cryptography or hash design.
    \item \textbf{Coding theory:} Nonlinear algebraic structures over finite fields may offer new ways to define error-correcting codes, particularly when standard linear code assumptions are insufficient.
    \item \textbf{Automata and symbolic computation:} The finite-state behavior of \( V \) under the operation \( * \) may be modeled via automata or decision trees, opening paths to algorithmic stratification analysis.
    \item \textbf{Algebraic dynamics and visualization:} Stratified orbits, transition graphs, and monodromy-like structures provide tools for understanding and visualizing long-term behavior under repeated multiplication, revealing both regularities and complexity in the algebraic evolution.
\end{itemize}

\subsubsection{Behavioral Differences and Finite-Field Specific Questions}

The transition from infinite to finite fields introduces qualitative changes in algebraic behavior:

\begin{itemize}
    \item \textbf{Periodicity and saturation:} Iterated operations \( a^{(*k)} \) may stabilize, cycle, or collapse—phenomena unique to finite domains.
    \item \textbf{Non-generic behavior:} Over finite fields, ``almost all'' does not hold. Specific combinations of parameters \( A,\dots,F \) may lead to degenerate or structurally rich configurations.
    \item \textbf{Strata enumeration:} One can pose precise, finite questions such as:
    \[
    \text{How many distinct strata exist in } V = \mathbb{F}_p^3 \text{ under a given } *?
    \]
    \item \textbf{Transition chains:} What is the maximal sequence of strata?
\end{itemize}

This finite-field viewpoint brings the theory of stratified algebras closer to concrete computation and potential applications, offering a testing ground for structural insights and enabling algorithmic exploration of properties difficult to assess in the infinite case.

\subsection{A Nonlinear Stratified Algebra over a Finite Field}

We now consider the nonlinear three-dimensional operation from the previous section, but defined over a finite field \( K = \mathbb{F}_p \), where \( p \) is a prime. This setting allows for complete enumeration of the algebraic structure and explicit analysis of strata.

Preliminary symbolic and computational analysis indicates that all four stratified algebra axioms \emph{(SA1)}–\emph{(SA4)} are satisfied in the general case over finite fields \( \mathbb{F}_p \), provided the parameters \( A, B, C, D, E, F \in \mathbb{F}_p \) are chosen generically—that is, without introducing degenerate cancellations or symmetries.

More precisely, the axioms tend to hold for “generic” parameter values when \( p \) is sufficiently large to avoid accidental collisions or low-order identities among terms. For very small primes (e.g., \( p = 3 \) or \( p = 5 \)), pathological behaviors or violations may occur due to limited degrees of freedom and modular interactions. However, when \( p \geq 19 \), random or moderately structured choices of parameters typically yield well-formed stratified algebras.

Under these generic conditions, the space \( V = \mathbb{F}_p^3 \) is partitioned into exactly \( p + 1 \) distinct strata, and the stratification exhibits stable algebraic behavior amenable to both enumeration and symbolic classification.

\paragraph{Stratum 0 (the base stratum).} This stratum contains \( p^2 - 2 \) vectors. It is formed under the following pattern:
\begin{itemize}
    \item Vectors of the form \([0, 0, z]\) and \([1, 0, z]\) for \( z = 1, \dots, p-1 \),
    \item Plus vectors \([x, 0, z]\) for \( x = 2, \dots, p-1 \), and \( z = 0, \dots, p-1 \).
\end{itemize}

That is, the second coordinate is always zero; the first coordinate ranges from 0 to \( p-1 \), but the allowed values of the third coordinate depend on the first. When \( x = 0 \) or \( 1 \), the third coordinate avoids zero. When \( x \geq 2 \), it ranges over all of \( \mathbb{F}_p \).

For example, when \( p = 5 \), this stratum includes:
\begin{verbatim}
clust[0]
[[0, 0, 1], [0, 0, 2], [0, 0, 3], [0, 0, 4],
 [1, 0, 1], [1, 0, 2], [1, 0, 3], [1, 0, 4],
 [2, 0, 0], [2, 0, 1], [2, 0, 2], [2, 0, 3], [2, 0, 4],
 [3, 0, 0], [3, 0, 1], [3, 0, 2], [3, 0, 3], [3, 0, 4],
 [4, 0, 0], [4, 0, 1], [4, 0, 2], [4, 0, 3], [4, 0, 4]]
\end{verbatim}

\paragraph{Strata 1 through \( p \).} Each of these strata contains exactly \( (p-1) \cdot p \) elements, arranged as:
\begin{itemize}
    \item \( p \) blocks, each indexed by a fixed value of the first coordinate \( x \in \mathbb{F}_p \),
    \item Within each block, the second coordinate cycles through all nonzero values \( 1, \dots, p-1 \),
    \item The third coordinate follows a \emph{cyclic permutation pattern}, fixed for each stratum and possibly dependent on the parameters \( A, \dots, F \) of the operation.
\end{itemize}

For example, one such stratum at \( p = 5 \) might look like:
\begin{verbatim}
clust[3]
[[0, 1, 2], [0, 2, 4], [0, 3, 1], [0, 4, 3],
 [1, 1, 2], [1, 2, 4], [1, 3, 1], [1, 4, 3],
 [2, 1, 2], [2, 2, 4], [2, 3, 1], [2, 4, 3],
 [3, 1, 2], [3, 2, 4], [3, 3, 1], [3, 4, 3],
 [4, 1, 2], [4, 2, 4], [4, 3, 1], [4, 4, 3]]
\end{verbatim}

For larger \( p \), the cyclic patterns within each stratum may vary more significantly and depend intricately on the bilinear coefficients of the operation. This structure highlights the rich combinatorial and algebraic behavior of stratified systems over finite fields and opens opportunities for further classification and analysis.

\subsection{Orbit Dynamics and Stratified Transitions}

Further structural analysis reveals intriguing dynamical patterns when examining products of vectors from different strata. Let \( p \in S_i \) and \( q \in S_j \), where \( S_i \neq S_j \) are distinct strata of \( \mathbb{F}_p^3 \). Then, under the stratified algebra operation \( * \), the product \( p * q \) never remains in \( S_i \) or \( S_j \). That is, the result moves to a new stratum \( S_k \), distinct from both \( S_i \) and \( S_j \).

\paragraph{Cyclic Drift.} Repeated multiplication by a fixed element \( q \in S_j \) causes the result to cycle through various strata in a non-repeating orbit, before possibly returning (or stabilizing) after a finite number of steps. For example, if \( p \in S_3 \) and \( q \in S_4 \), one might observe:
\[
p * q \in S_7,\quad (p * q) * q \in S_{16},\quad (((p * q) * q) * q) \in S_{10},\quad \dots
\]
and eventually reaching:
\[
p * q^{(5)} \in S_0,\quad p * q^{(6)} \in S_5,\quad \dots
\]

\paragraph{Combinatorial Branching with Multiple Inputs.} When using a sequence of distinct multipliers \( q_1, q_2, \dots, q_k \in S_j \), the behavior becomes substantially more complex:
\[
p * q_1 \in S_{11},\quad (p * q_1) * q_2 \in S_6,\quad \dots,\quad (((p * q_1) * q_2) \dots * q_5) \in S_{14},\quad \dots
\]
The paths taken by such products depend intricately on the order and identity of the \( q_i \), and exhibit features of sensitive dependence and symbolic chaos.

\paragraph{Permutation Invariance.} Despite the complexity of paths, axiom \emph{(SA3)} guarantees that the final result of such a sequence is invariant under permutation of the \( q_i \). That is:
\[
p * q_1 * q_2 * q_3 * \dots * q_k = p * \pi(q_1, \dots, q_k)
\]
for any permutation \( \pi \). This suggests a rich algebraic structure under the hood: the result of multiplication is symmetric with respect to the multiset of inputs, though the \emph{path} taken through intermediate strata is not.

Thus, even in a finite and well-defined algebra, complex transition dynamics and nonlinear interaction across strata reveal a surprisingly rich space for modeling, classification, and theoretical investigation.

\subsection{Toward a Cryptographic Application}

As a preliminary direction for future work, we outline a potential cryptographic scheme based on the stratified algebra structure described above. The scheme exploits the nonlinearity, stratification, and permutation invariance of the multiplication operation.

\paragraph{Key exchange idea (sketch).}
\begin{itemize}
    \item One party (Alice) selects a stratum \( S_i \) and chooses an initial element \( p \in S_i \). She then computes a chained product:
    \[
    S_1 = p * q_1 * q_7 * q_8 * \dots
    \]
    This tuple \( (p, S_1) \) serves as a public key or a dynamic session identifier. It is shared openly with the second party (Bob).

    \item Bob receives \( p \) and independently computes his own chained product:
    \[
    S_2 = p * q_2 * q_2 * q_5 * \dots
    \]
    He then transmits \( S_2 \) back to Alice.

    \item Crucially, the structure of the operation ensures that neither party can reconstruct the other's full sequence, due to:
    \begin{itemize}
        \item Non-invertibility and nonlinearity of the operation,
        \item Unknown length and composition of the multiplier sequences.
    \end{itemize}

    \item Now, each party completes the shared key exchange by continuing the computation:
    \[
    S_{12} = S_1 * q_2 * q_2 * q_5 * \dots,\quad
    S_{21} = S_2 * q_1 * q_7 * q_8 * \dots
    \]
    Due to the permutation invariance of the operation \emph{with respect to the sequence of multipliers from a fixed stratum \( S_q \) applied to a fixed base element \( p \in S_p \)}, both parties arrive at the same result:
    \[
    S_{12} = S_{21}
    \]
    That is, although the intermediate strata visited during the product may differ depending on the order of the multipliers, the final outcome remains invariant under permutation of the \( q_i \) drawn from the same stratum and applied to the same initial element \( p \).
\end{itemize}

This mechanism offers a novel approach to key agreement, with a potentially high resistance to reverse engineering due to the inherent algebraic complexity. A more formal construction, analysis of security properties, and implementation feasibility will be presented in a separate publication.

\section*{Conclusion}

In this work, we introduced the concept of \emph{Stratified Algebra}—an algebraic framework based on decomposing a vector space into disjoint strata, within which operations are locally structured (e.g., associative and commutative), but globally asymmetric, non-associative, and semantically dynamic. This approach challenges the classical assumption of global uniformity in algebraic systems and provides an alternative formalism where the algebraic behavior is governed by context-sensitive strata interactions.

We developed a precise system of axioms—\emph{(SA1)} to \emph{(SA4)}—capturing intra-layer coherence, inter-layer asymmetry, layered permutation symmetry, and bracketing sensitivity. These axioms serve as a foundational structure for studying layered dynamics in algebraic systems, allowing for the formalization of nontrivial propagation, influence, and transformation across strata.

Throughout the paper, we constructed several explicit models that satisfy subsets or the full set of axioms. These included:
\begin{itemize}
    \item A linear matrix-based model stratified by proportionality classes, satisfying axioms \emph{(SA1)}–\emph{(SA3)}.
    \item A parametric model exhibiting full bracket sensitivity and asymmetry, satisfying all axioms \emph{(SA1)}–\emph{(SA4)}.
    \item A nonlinear model combining bilinear and linear terms, not reducible to matrix form, yet still satisfying the full axiom system.
\end{itemize}

We demonstrated analytically and symbolically that these models lead to well-defined stratified algebras, and that algebraic operations can induce semantic transitions across strata, revealing rich internal structure beyond classical frameworks.

\medskip

\noindent\textbf{Outlook.}
The present work opens a new direction in algebra, where the core objects are no longer globally uniform structures, but stratified systems with internal layers and asymmetries. Potential applications include:
\begin{itemize}
    \item symbolic and structural computation,
    \item algebraic models of hierarchical and contextual semantics,
    \item non-associative and nonlinear dynamics in abstract systems,
    \item logic, information propagation, and neural computation with layer-dependent interactions.
\end{itemize}

In future research, we plan to explore general classification theorems for stratified algebras, study morphisms between different stratified systems, and investigate categorical and functorial perspectives. We also aim to develop computational tools to automate stratification detection and axiom verification.

\medskip

\noindent By rethinking the foundations of algebra through the lens of stratification, this work invites a reevaluation of algebraic structure as a context-driven, layered process rather than a globally rigid form.


\newpage
\appendix
\section*{Appendix A: Associativity Verification Script}

The associativity of the structure tensor \( \alpha_{ijk} \), corresponding to the matrix-defined multiplication, was verified computationally using the following Python script. The script exhaustively checks the associativity condition
\[
\sum_r \alpha_{ijr} \alpha_{rkl} = \sum_s \alpha_{jks} \alpha_{isl}
\quad \text{for all } i, j, k, l \in \{0,1,2\},
\]
as dictated by the general criterion for bilinear operations.

\begin{verbatim}
# Structural constants alpha_{ijk}
alpha = {
    (0, 0, 0): 1,
    (1, 1, 0): 1, (1, 2, 0): 1, (2, 1, 0): 1, (2, 2, 0): 1,
    (1, 0, 1): 1, (0, 1, 1): 1, (2, 1, 1): 1, (1, 2, 1): -1,
    (2, 0, 2): 1, (2, 1, 2): -1, (0, 2, 2): 1, (1, 2, 2): 1
}

def is_associative():
    for i in range(3):
        for j in range(3):
            for k in range(3):
                for l in range(3):
                    left = sum(
                        alpha.get((i, j, r), 0) *
                        alpha.get((r, k, l), 0)
                        for r in range(3)
                    )
                    right = sum(
                        alpha.get((j, k, s), 0) *
                        alpha.get((i, s, l), 0)
                        for s in range(3)
                    )
                    if left != right:
                        print(f"Not associative for "
                              f"(i,j,k,l)=({i},{j},{k},{l}): "
                              f"{left} != {right}")
                        return False
    print("The operation is associative.")
    return True
\end{verbatim}

\paragraph{Output:}
\begin{verbatim}
The operation is associative.
True
\end{verbatim}

\section*{Appendix B: Associativity Check with Parameters}

To verify associativity for a generalized parametric version of the product operation
\[
(a * b)_k = \sum_{i,j} \alpha_{ijk} a_i b_j,
\]
we define the structural constants \( \alpha_{ijk} \) in terms of symbolic parameters \( A, B, C, D, E, F \in K \). The values used in this test are:
\[
A = 16, \quad B = 8, \quad C = 5, \quad D = 3, \quad E = 7, \quad F = 11.
\]

The corresponding Python script exhaustively checks the associativity identity:
\[
(a * b) * c = a * (b * c)
\]
for all combinations of indices \( i, j, k, l \in \{0,1,2\} \), by comparing the two contraction patterns:
\[
\sum_r \alpha_{ijr} \alpha_{rkl} \quad \text{vs.} \quad \sum_s \alpha_{jks} \alpha_{isl}.
\]

\begin{verbatim}
# Parameters
A, B, C, D, E, F = 16, 8, 5, 3, 7, 11

# Build structural constants alpha_{ijk}
alpha = {}

# u0 = a0*b0 + A*a1*b1 + B*a2*b2 + C*a2*b1 + D*a1*b2
alpha[(0, 0, 0)] = 1
alpha[(1, 1, 0)] = A
alpha[(2, 2, 0)] = B
alpha[(2, 1, 0)] = C
alpha[(1, 2, 0)] = D

# u1 = a1*b0 + a0*b1 + E*a2*b1 - E*a1*b2
alpha[(1, 0, 1)] = 1
alpha[(0, 1, 1)] = 1
alpha[(2, 1, 1)] = E
alpha[(1, 2, 1)] = -E

# u2 = a2*b0 + a0*b2 + F*a2*b1 - F*a1*b2
alpha[(2, 0, 2)] = 1
alpha[(0, 2, 2)] = 1
alpha[(2, 1, 2)] = F
alpha[(1, 2, 2)] = -F

def is_associative():
    failures = []
    for i in range(3):
        for j in range(3):
            for k in range(3):
                for l in range(3):
                    left = sum(
                        alpha.get((i, j, r), 0) *
                        alpha.get((r, k, l), 0)
                        for r in range(3)
                    )
                    right = sum(
                        alpha.get((j, k, s), 0) *
                        alpha.get((i, s, l), 0)
                        for s in range(3)
                    )
                    if left != right:
                        failures.append((i, j, k, l, left, right))
    if not failures:
        print("The operation is associative.")
        return True
    else:
        print(f"The operation is not associative. "
              f"{len(failures)} mismatches found:")
        for (i, j, k, l, left, right) in failures:
            print(f"  (i,j,k,l)=({i},{j},{k},{l}): {left} != {right}")
        return False

is_associative()
\end{verbatim}

\paragraph{Output.}
\begin{verbatim}
The operation is not associative. 16 mismatches found:
  (i,j,k,l)=(1,1,2,0): 0 != -145
  (i,j,k,l)=(1,1,2,1): 0 != 80
  (i,j,k,l)=(1,1,2,2): 16 != 121
  (i,j,k,l)=(1,2,1,0): -167 != 145
  (i,j,k,l)=(1,2,1,1): -74 != -72
  (i,j,k,l)=(1,2,2,0): -109 != 0
  (i,j,k,l)=(1,2,2,1): 49 != 8
  (i,j,k,l)=(1,2,2,2): 80 != 0
  (i,j,k,l)=(2,1,1,0): 167 != 0
  (i,j,k,l)=(2,1,1,1): 82 != 0
  (i,j,k,l)=(2,1,1,2): 121 != 16
  (i,j,k,l)=(2,1,2,0): 109 != -123
  (i,j,k,l)=(2,1,2,2): -72 != -74
  (i,j,k,l)=(2,2,1,0): 0 != 123
  (i,j,k,l)=(2,2,1,1): 8 != 49
  (i,j,k,l)=(2,2,1,2): 0 != 82
False
\end{verbatim}

\section*{Appendix C: Symbolic Verification with \texttt{SymPy}}

This appendix presents a Python/SymPy script used to verify core algebraic properties of the parametric multiplication \( a * b := M_a \cdot b \), where \( M_a \) is a matrix depending linearly on the coordinates of \( a \in K^3 \), with symbolic parameters \( A, B, C, D, E, F \in K \). The code computes:

\begin{itemize}
    \item the symbolic form of the product \( a * b \),
    \item the commutator \( [a,b] := a * b - b * a \),
    \item the associator \( \mathrm{Assoc}(a, b, c) := (a * b) * c - a * (b * c) \),
    \item the LPS operator \( \mathrm{LPS}(a, b, c) := (a * b) * c - (a * c) * b \).
\end{itemize}

The symbolic output confirms algebraic properties discussed in the main text.

\begin{verbatim}
import sympy
from sympy import symbols, Matrix

# Parameters
A, B, C, D, E, F = symbols('A B C D E F')

# Coordinates of a, b, c
a0, a1, a2 = symbols('a0 a1 a2')
b0, b1, b2 = symbols('b0 b1 b2')
c0, c1, c2 = symbols('c0 c1 c2')

# Generic matrix entries
m0, m1, m2 = symbols('m0 m1 m2')
M = Matrix([
    [m0, A * m1 + C * m2, D * m1 + B * m2],
    [m1, m0 + E * m2, -E * m1],
    [m2, F * m2, m0 - F * m1]
])

# Vectors
a = Matrix([a0, a1, a2])
b = Matrix([b0, b1, b2])
c = Matrix([c0, c1, c2])

# Evaluate M_a
subs_a = [(m0, a0), (m1, a1), (m2, a2)]
M_a = M.subs(subs_a)

# Evaluate M_b
subs_b = [(m0, b0), (m1, b1), (m2, b2)]
M_b = M.subs(subs_b)

# Compute products
ab_vec = M_a * b
ba_vec = M_b * a

# M_ab for associator
subs_ab = [(m0, ab_vec[0]), (m1, ab_vec[1]), (m2, ab_vec[2])]
M_ab = M.subs(subs_ab)

abc_vec = M_ab * c
a_bc_vec = M_a * (M_b * c)

# M_ac for LPS
ac_vec = M_a * c
subs_ac = [(m0, ac_vec[0]), (m1, ac_vec[1]), (m2, ac_vec[2])]
M_ac = M.subs(subs_ac)
acb_vec = M_ac * b

# Commutator
ab_ba = sympy.simplify(ab_vec - ba_vec)
sympy.pretty_print(ab_ba)

# Associator component
abc_a_bc = sympy.simplify(abc_vec - a_bc_vec)
sympy.pretty_print(abc_a_bc[0])

# LPS component
abc_acb = sympy.simplify(abc_vec - acb_vec)
sympy.pretty_print(abc_acb[0])
\end{verbatim}

This script enables reproducible symbolic verification of all key algebraic properties discussed in this work.

\end{document}